\documentclass[12pt]{amsart}

\usepackage{preamble}

\title{
Zero cycles on products of elliptic curves over local fields with supersingular reduction
} 
\author{Alejandro De Las Pe\~nas Casta\~no}
\address{Department of Mathematics, University of Virginia, Charlottesville, VA 22904}
\email{ad7ag@virginia.edu}

\begin{document}

\begin{abstract}
  For a product $E_1\times E_2$ of two elliptic curves over a $p$-adic field with good supersingular reduction, we produce infinitely many rational equivalences in the Chow group $\CH_0(X)$ of zero cycles via genus 2 covers of $E_1$ and $E_2$. We use this to obtain evidence for a conjecture of Colliot-Th\'el\`ene about the structure of the Albanese kernel. 
\end{abstract}

\maketitle
\section*{Introduction}

Let $X$ be a smooth projective variety defined over a field $K$. A fundamental invariant of $X$ is the Chow group $\CH_0(X)$ of 0-cycles modulo rational equivalence. It is a generalization of the Picard group of curves and similarly comes equipped with an Abel-Jacobi map
\[
    \mathrm{alb_X}:F^1(X)\rightarrow\mathrm{Alb}_X(K),
\]
called the \emph{Albanese} map, where the domain $F^1(X)\subseteq\CH_0(X)$ is the subgroup of 0-cycles of degree zero and $\mathrm{Alb}_X$ is an abelian variety with the same universal property the Jacobian of a curve has, namely, when $X$ possesses a $k$-rational point, there is a map $X\rightarrow\mathrm{Alb}_X$ through which every map from $X$ into an abelian variety factors uniquely. 


As opposed to the case of curves, the Albanese map need not be injective and its kernel, denoted by $F^2(X)$, is known to vary wildly. For instance it is finite when $K$ is finite \cite{KatoSaito1983}, but can be uncountable, when $K=\mathbb{C}$ \cite{Mumford1969}. When $K$ is a number field the problem of determining $F^2(X)$ remains wide open. In this direction, there are far-reaching conjectures by Beilinson and Bloch \cite{Beilinson1985,Bloch1984} that suggest that $F^2(X)$ is finite.

A historically fruitful testing ground for such conjectures is the case when $K$ is a $p$-adic field, i.e. a finite extension of $\mathbb{Q}_p$. In this case, we have the following conjecture of Colliot-Th\'el\`ene, later refined by Raskind and Spiess.

\begin{letterconjecture}\label{conj:CT}
  (\cite[Conjecture 1.4]{Colliot-Thelene1995}, \cite[Conjecture 3.5.4]{RaskindSpiess2000}) Let $X$ be a smooth projective geometrically connected variety defined over a $p$-adic field $K$. Then $F^2(X)\cong D\oplus F$, where $D$ is a divisible group and $F$ is a finite group.
\end{letterconjecture}

An important advancement in this direction is the work of Saito and Sato \cite{SaitoSato2010} where they show, under some mild technical conditions, that the larger group $F^1(X)$ is the direct sum of a finite group and a group that is divisible by every integer coprime to $p$. Thus the key unresolved aspect of Colliot-Th\'el\`ene's conjecture is proving that the quotients $F^2(X)/p^n$ are finite for all $n\geq1$, along with determining the asymptotic behavior of their sizes as $n\rightarrow\infty$.

The full conjecture has been established only in very limited cases and for specific classes of varieties. In 2000, Raskind and Spiess proved the conjecture for varieties that are products of curves whose jacobians have a mixture of good ordinary reduction and split multiplicative reduction. Later, Gazaki and Leal were able to prove the conjecture for $X=E_1\times E_2$ where at most one curve was allowed to have supersingular reduction \cite{GazakiLeal2022}. The case when both curves have supersingular reduction is still unknown and is the case of study of this paper.

The starting point of the above results is to show that $F^2(E_1\times E_2)$ is isomorphic to a certain $K$-group $K(K;E_1,E_2)$ called the \emph{Somekawa} $K$-group \cite{Somekawa1990}. It is known that $F^2(E_1\times E_2)$ is generated by the 0-cycles on $E_1\times E_2$ that come from intersecting the pullbacks of 0-cycles on $E_1$ and $E_2$. The bilinear nature of these generators yields a surjection
\[
    \bigoplus_{L/K\,\text{finite}}E_1(L)\otimes_\mathbb{Z} E_2(L) \epi F^2(E_1\times E_2).
\]
Raskind and Spiess were able to compute the kernel and show that the quotient is precisely the Somekawa $K$-group $K(K;E_1,E_2)$. This quotient is visibly generated by equivalence classes of simple tensors of the form $P_1\otimes P_2$ with $P_i\in E_i(L)$ as $L$ ranges over all finite extensions of $K$; such a class is called an $L/K$-\emph{symbol} and is denoted by $\{P_1,P_2\}_{L/K}$.

From this description, one can see why determining the size of $F^2(X)/p^n$ is difficult. Even showing that $F^2(X)/p$ is finite is still very hard. Indeed, each summand $E_1(L)/p\otimes E_2(L)/p$ of $K(K;E_1,E_2)/p$ can be quite large, requiring many relations to offset the shear number of generators. Furthermore, the structure of each summand depends sensitively on both the ramification of $K$ and on the reduction type of the elliptic curves.

When $E_1$ and $E_2$ have good reduction and the absolute ramification index $e$ of $K$ is small, i.e. $e<p-1$, it is expected that
\[
    K(k;E_1,E_2)/p=0.
\]
This expectation comes from the vanishing of a certain cycle map on $K(K;E_1,E_2)/p$ \cite[Corollary 8.1]{Gazaki2018}. So in order to prove a statement of this type, one must be able to produce enough relations in $K(K;E_1,E_2)$ to cancel the generators of the Somekawa $K$-group modulo $p$. In this paper, we focus on the ``first level'' of $K(K;E_1,E_2)/p$, that is the subgroup of $K/K$-symbols.

There are two types of relations in $K(K;E_1,E_2)$ one can draw from: the first, called the \emph{Projection Formula} (PF), and the second, called \emph{Weil Reciprocity} (WR). The former type of relation encodes how points on $E_1\times E_2$, when viewed over different fields extensions, interact with each other via norm and restriction maps. In fact, these relations were a key component of the previously mentioned methods of Raskind-Spiess and Gazaki-Leal. However, when both elliptic curves have supersingular reduction, PF relations fail to produce enough cancellations.

In this paper we produce infinitely many distinct WR relations in $K(K;E_1,E_2)$. For these to be useful, we need a way to classify them. One can show that the supersingularity assumption implies that every $\{P_1,P_2\}_{K/K}$ will be congruent modulo $p$ to a $K/K$-symbol $\{\widehat{P}_1,\widehat{P}_2\}_{K/K}$ where $\widehat{P}_i$ lies in the formal group $\widehat{E}_i(\mathfrak{m})$ (see Section \ref{sec:ssing}). We can hence distinguish WR relations with the \emph{signature} of $\{P_1,P_2\}_{K/K}$, defined to be the pair of integers $(n_1,n_2)$ where $n_i$ is the valuation of $\wh{P}_i$ viewed as an element of the formal group (cf. Definition \ref{def:signature-of-formal-point-mod-p}). The first main result of this paper is the following.

\begin{lettertheorem}\label{thm:main-theorem-intro}
  Let $p>3$ be prime. Let $K$ be a $p$-adic field that is totally ramified over $\Q_p$. Let $E_1$ and $E_2$ be elliptic curves defined over $K$ with full $K$-rational 2-torsion and supersingular reduction. Then for any $n>0$, there exists a WR relation $\{P_1,P_2\}_{K/K}=0$ in $K(K;E_1,E_2)$ of signature $(n,n)$.
\end{lettertheorem}

The WR relations produced in Theorem \ref{thm:main-theorem-intro} come from genus 2 covers of $E_1$ and $E_2$. More precisely, there is a WR relation in $K(K;E_1,E_2)$ for every curve $C$ over $K$, whose Jacobian can be mapped into $E_1\times E_2$, and a choice of rational function $f\in K(C)^\times$. In this paper, we use a construction due to J. Scholten of a genus 2 curve $C$ and explicit nonconstant maps $\f_1:C\ra E_1$ and $\f_2:C\ra E_2$. Since $C$ is hyperelliptic it has particularly simple principal divisors, so the choice of rational function can be made flexibly. Armed with a suitable choice of $C,\f_1,\f_2$ and $f$, we can then $p$-adically approximate the resulting WR relation to determine its signature.

This paper is organized into three sections. In section \ref{sec:prelim}, we review all the basic theory we will need about formal groups of elliptic curves, supersingular reduction, and the Somekawa $K$-group for products of elliptic curves. Along the way we define the signature of a $K/K$-symbol and determine the type and quantity of WR relations one needs to show the vanishing of all $K/K$-symbols modulo $p$. In section \ref{sec:scholten_curves} we review the construction of the hyperelliptic curves that map to $E_1\times E_2$, compute the WR relations they produce, and use them to prove Theorem \ref{thm:main-theorem-intro}. Lastly, in section \ref{sec:computations}, we focus on the case of quadratic ramified extensions $K$ and present computational evidence that the genus 2 spans of section \ref{sec:scholten_curves} produce the necessary WR relations to show that all $K/K$-symbols vanish modulo $p$.

\section*{Notation}

Throughout, $p$ will be a prime larger than 3 and $K$ will be a finite extension of $\Q_p$, the field of $p$-adic numbers. The $p$-adic field $K$ comes equipped with a unique normalized valuation we denote by $\nu$. Let $\O$ be its valuation ring, $\m$ its unique maximal ideal, and $k=\O/\m$ its residue field. Sometimes we will fix a uniformizer $\pi$, so in particular $\m=(\pi)$ and $\nu(\pi)=1$. The reduction map $\O\ra k$ will be denoted by $\alpha\mapsto\ol{\a}$. The ramification index of $K$ over $\Q_p$ will be denoted by $e$ so in particular $\nu(p)=e$. If we are working over different base fields at the same time we label the above objects with the corresponding field, e.g. $\O_K,\m_K,k_K$, etc.

We will denote elliptic curves with $E$, these are smooth projective curves of genus 1 defined over $K$ with a fixed $K$-rational point we denote by $O$ that serves as the identity of the abelian group of $K$-rational points $E(K)$. Given a point $P\in E(K)$, we denote its coordinates by $x_P$ and $y_P$, or by $x(P)$ and $y(P)$. We denote the function field of $E$ by $K(E)$ and given a closed point $Q$ of $E$, we denote the residue field of $Q$ by $K(Q)$.

Elliptic curves defined over $p$-adic fields have minimal Weierstrass equations of the form
\[
  y^2 = x^3 + a_2 x^2 + a_4 x + a_6,
\]
where $a_i\in\O$. If we reduce each coefficient modulo $\m$, we obtain an elliptic curve $\ol{E}$ defined over the residue field $k$. There is a well-defined reduction map $E(K)\ra\ol{E}(k)$ which we also denote by $P\mapsto\ol{P}$.

Given any abelian group $A$, such as $E(K)$ or the formal group of $E$, and an integer $n$, we denote the $n$-torsion subgroup of $A$ by $A[n]$ and the cokernel of multiplication-by-$n$ by $A/n$. The equivalence class of $a\in A$ in $A/n$ will be denoted by $[a]$.

\section*{Acknowledgements}

The author would like to thank his academic siblings T. Jaklitsch and M. Wills for helpful discusions,  H. Saad for pointing out a significant simplification to one of the lemmas, and the author's advisor E. Gazaki for their support and guidance. This paper was written while the author was partially supported by E. Gazaki's NSF grant DMS-2305231.

\section{Preliminaries}\label{sec:prelim}

\subsection{Formal Groups}\label{sec:formal-groups}

In this section, we review the necessary background on formal groups and prove some structural results for $\E(\m)/p$. For more details, we refer the reader to \cite[Chapter IV]{Silverman2013} or \cite[Part II, Chapter IV]{Serre1992}.

Let $E/K$ be an elliptic curve and let $\E(\m)$ be its associated formal group. The underlying set of $\E(\m)$ is the maximal ideal $\m\subset\O$ which induces the filtration
\[
  \E(\m) \supseteq \E(\m^2) \supseteq \E(\m^3) \supseteq \cdots,
\]
where the underlying set of $\E(\m^n)$ is $\m^n$.

Now lets assume that $E$ has good reduction, i.e. the reduced curve $\bar{E}$ is non-singular over $k$. It is well-known \cite[Proposition VII.2.1]{Silverman2013} that there is an exact sequence:
\begin{equation}\label{eq:reduction-sequence}
  \begin{tikzcd}
    0 \arrow[r] & E_1(K) \arrow[r] & E(K) \arrow[r] & \bar{E}(k) \arrow[r] & 0
  \end{tikzcd}
\end{equation}
where $E_1(K)$ is the kernel of the reduction map. In fact, this kernel is isomorphic to the formal group of $E$ via
\[
  E_1(K)\lra\E(\m)\quad\text{defined by}\quad P=(X,Y)\mapsto -X/Y.
\]
We will implicitly make this identification throughout.

The multiplication-by-$p$ map in $\E(\m)$ is compatible with the filtration in the following sense.

\begin{theorem}\label{thm:image_of_mult_by_p_map_formal_groups}(Theorem 4, pg. 119 of \cite{Serre1992}) If $r>e/(p-1)$, then $p\E(\m^r)=\E(\m^{r+e})$. In particular, if $e<p-1$, then $p\E(\m)=\E(\m^{1+e})$.
\end{theorem}

\begin{remark}\label{rm:Fp-vector-space-structure-of-Emodp}
  If $e<p-1$, then $\E(\m)/p$ is a $k$-vector space and Theorem \ref{thm:image_of_mult_by_p_map_formal_groups} determines its structure as follows. The filtration of $\E(\m)$ reduces to a finite filtration
\[
  \E(\m)/p \supseteq \E(\m^2)/p \supseteq\cdots\supseteq \E(\m^e)/p \supseteq 0
\]
of length $e$. Each piece $\E(\m^i)/p$ is naturally an $\F_p$-vector space and the graded components are all the additive group of $k$ (see for example Corollary 1, pg 118 of \cite{Serre1992}). That is
\[
  \frac{\E(\m^i)/p}{\E(\m^{i+1})/p}\cong k^+.
\]
If $K$ has inertia degree $f$ over $\Q_p$, then $k^+$ is an $f$-dimensional $\F_p$-vector space and thus $\E(\m)/p$ has dimension $ef=[K:\Q_p]$. Furthermore, if for each $i=1,\ldots,e$, one chooses elements $\b_{i,1},\ldots,\b_{i,f}\in\E(\m^i)$ whose classes in $\E(\m^i)/\E(\m^{i+1})$ form an $\F_p$-basis, then the set $\{[\b_{i,j}]\mid 1\leq i\leq e,1\leq j\leq f\}$ is an $\F_p$-basis of $\E(\m)/p$. 
\end{remark}

For $e<p-1$, one can talk about valuations of elements in $\E(\m)/p$. Given a nonzero class $[\a]$ in $\E(\m)/p$, the integer $\nu(\a)$ is independent of the choice of representative. To see this, suppose $\a'$ is some other representative so $\a=\a'+p\b$ for some $\b\in\E(\m)$. Since $[\a]\neq0$, Theorem \ref{thm:image_of_mult_by_p_map_formal_groups} implies $\nu(\a)\leq e$ and similarly $\nu(\a')\leq e$. Since addition in $\E(\m)$, in fact any formal group, is given by $\E(X,Y)=X+Y+XY F(X,Y)$ for some $F\in\O[\![X,Y]\!]$ (see \S6, Chapter IV, Part II of \cite{Serre1992}), we have
\[
  \E(\a',p\b) = \a'+ p\b + p\a'\b A,
\]
for some $A\in\O$. Therefore
\[
  \nu(\a)=\nu\big(\a'+p(1+\a'A)\b\big)=\nu(\a'),
\]
since $p(1+\a'A)\b\in\E(\m^{1+e})$ has larger valuation than $\a'$ by Theorem \ref{thm:image_of_mult_by_p_map_formal_groups}.

Roughly speaking, the above argument shows that the usual valuation $\nu$ on $\E(\m)$ descends to $\E(\m)/p$. However one must be careful in declaring it a valuation. Firstly, any nonzero class $[\a]$ will have $\nu([\a])\leq e$, and secondly, $\nu([0])$ depends on its representative which can be chosen to have arbitrarily large valuation.

However, this ``valuation'' is useful for computations in $\E(\m)/p$, so we formalize it below.

\begin{definition}\label{def:signature-of-formal-point-mod-p}
  Let $K$ be a $p$-adic field with $e<p-1$. Let $E$ be an elliptic curve over $K$. The \emph{signature} of a class $[\a]\in\E(\m)/p$ is defined as
  \[
    \sS([\a]):=
    \begin{cases}
      \nu(\a) &\text{if}\; \a\not\equiv 0\spmod{p}\\
      +\infty &\text{otherwise}.
    \end{cases}
  \]
\end{definition}

As expected, the signature has several of the same properties that the valuation $\nu$ has.

\begin{proposition}\label{prop:properties-of-signature}
  Let $K$ be a $p$-adic field with $e<p-1$ and $E$ an elliptic curve over $K$. Let $[\a],[\b]\in\E(\m)/p$ be nonzero classes. Then we have the following properties.
  \begin{enumerate}[label=(\roman*)]
    \item\label{item:signature-is-non-archimedean} In general, $\sS([\a]+[\b])\geq\min\{\sS([\a]),\sS([\b])\}$, with equality if $\sS([\a])\neq\sS([\b])$.
    \item If $u\in\Z\cap O^\times$, then $\sS(u[\a])=\sS([\a])$.
  \end{enumerate}
\end{proposition}

\begin{proof}
  Suppose $\nu(\a)=m$ and $\nu(\b)=n$, and write $\a=a_m\pi^m+a_{m+1}\pi^{m+1}+\cdots$ and $\b=b_n\pi^n+a_{n+1}\pi^{n+1}+\cdots$, with $a_i,b_j\in\O$. The sum of $\a$ and $\b$ in $\E(\m)$ is given by
  \begin{align*}
    \E(\a,\b)
    & = (a_m\pi^m + \cdots) + (b_n\pi^n+\cdots) + (a_m b_n \pi^{m+n}F(\a,\b)+\cdots),\\
    & = a_m\pi^m + b_n\pi^n + O(\pi^{m+1}).
  \end{align*}
  The first claim follows immediately. Finally, since multiplication-by-$u$ in $\E(\m)$ is multiplication-by-$u$ in $\m$ up to higher order terms (cf. \cite[IV.2.3]{Silverman2013}), there is a power series $G\in T^2\O[\![T]\!]$ such that $u\a$ as an element of $\E(\m)$ has $\pi$-adic expansion
  \[
    u\a + G(\a) = u(a_m \pi^m + \cdots ) + O(\pi^{2m}) = ua_m \pi^m + \cdots.
  \]
  Since $u$ is a unit, then $\nu(u\a)=m$. 
\end{proof}

\subsection{Elliptic Curves with Supersingular Reduction}\label{sec:ssing}

In this section, we focus on elliptic curves with supersingular reduction and identify the group $\E(\m)/p$ with $E(K)/p$ for such curves, this last quotient being a key component of the Somekawa $K$-group modulo $p$. We begin by briefly reviewing the definition of supersingular elliptic curves, of which there are a wealth of equivalent definitions. Below we list the two definitions we will need.
\begin{definition}\label{def:ssing}
  Let $E$ be an elliptic curve defined over $K$ with good reduction. We say that $E$ has \emph{supersingular reduction} if the reduced curve $\ol{E}$ defined over $k$ is supersingular, that is, it satisfies one of the following equivalent conditions:
  \begin{enumerate}
    \item $\ol{E}\left(\ol{k}\right)[p]=0$.
    \item\label{item:Duering} (Duering's criterion) If $\ol{E}$ is defined by a Weierstrass equation of the form $y^2=f(x)$, where $f(x)\in k[x]$ is a separable cubic polynomial, then the coefficient of $x^{p-1}$ of $f(x)^{(p-1)/2}$ is zero.
  \end{enumerate}
\end{definition}

\begin{remark}\label{rm:about-ssing-red} (about supersingular reduction)
  \begin{enumerate}
    \item Since $\ol{E}(k)$ is a finite abelian group, Lagrange's theorem guarantees that $E$ has supersingular reduction if and only if $|\ol{E}(k)|$ is coprime to $p$.
    \item If $p>3$ and $k=\F_p$, then $|\ol{E}(k)|=p+1$. See \cite[Exercise 5.10]{Silverman2013}.
    \item If $E$ has supersingular reduction and $L$ is a finite extension of $K$, then the base change $E_L$ also has supersingular reduction.
    \item If $E$ is in Legendre form $E_\l:y^2=x(x-1)(x-\l)$ for some $\l\neq 0,1$, then Duering's criterion can be restated as follows: $E_\l$ is supersingular if and only if $\l$ is a root of the following polynomial:
    \[
      H_p(T):=\sum_{i=0}^{m}\binom{m}{i}^2 T^i,\qquad m=\tfrac{p-1}{2}.
    \]
    See for example \cite[V.4.1]{Silverman2013}.
  \end{enumerate}
\end{remark}

\begin{lemma}\label{lemma:formal-group-mod-p}
  Let $E$ be an elliptic curve over $K$ with supersingular reduction. Then
  \[
    E(K)/p\cong \E(\m)/p.
  \]
  This isomorphism is induced by the inclusion $\E(\m)\hookrightarrow E(K)$.
\end{lemma}
\begin{proof}
  The reduction exact sequence \eqref{eq:reduction-sequence} fits into the following commutative diagram with exact rows and columns:
  \[
    \begin{tikzcd}
      0 \arrow[r] & \E(\m) \arrow[r] \arrow[d,"p"] & E(K) \arrow[r] \arrow[d,"p"]  & \ol{E}(k) \arrow[r] \arrow[d,"p"] & 0 \\
      0 \arrow[r] & \E(\m) \arrow[r] & E(K) \arrow[r] & \ol{E}(k) \arrow[r] & 0
    \end{tikzcd}
  \]
  The Snake Lemma yields the exact sequence
  \[
    \begin{tikzcd}
      \ker(\text{multiplication-by-$p$}) \arrow[r] & \E(\m)/p \arrow[r] & E(K)/p \arrow[r] & \ol{E}(k)/p.
    \end{tikzcd}
  \]
  Since $E$ has supersingular reduction, then the multiplication-by-p map on the reduced curve is an injective map of finite groups and hence bijective. Thus the above exact sequence has zeros at its edges, and thus the map $\E(\m)/p\ra E(K)/p$ induced by the inclusion is an isomorphism.

\end{proof}

\begin{remark}
  The diagram in the proof of Lemma \ref{lemma:formal-group-mod-p} is preserved under base change to any finite extension $L$ of $K$. Since $\ol{E_L}(k_L)[p]$ is still zero, then the above proof can be upgraded to show an isomorphism $E/p \cong \E/p$ of Mackey functors.
\end{remark}

It is possible to realize the inverse map of the inclusion $\E(\m)/p\hookrightarrow E(K)/p$ as multiplication by a fixed integer. Let $n$ be the order of the finite group $\ol{E}(k)$. Then for any point $P\in E(K)$, the point $nP$ is in the kernel of the reduction map, i.e. $nP\in\E(\m)$. Since $n$ is coprime to $p$, there are $m,m'\in\Z$ such that $1=mN+m'p$. Thus
\begin{equation}\label{eq:formal-point-associated-to-P}
  P = mnP+m'pP \equiv mnP \pmod{pE(K)},
\end{equation}
where $mnP\in\E(\m)$.

\begin{definition}\label{def:associated-formal-point}
  Let $E$ be an elliptic curve over $K$ with supersingular reduction. Let $n=|\ol{E}(k)|$ and fix the smallest positive integer $m$ such that $mn\equiv1\pmod{p}$. For $P\in E(K)$, we define the \emph{associated formal point} of $P$ to be $\wh{P}:=mnP\in \E(\m)\subset E(K)$.
\end{definition}

As discussed in the previous paragraph, we obtain a well-defined map
\[
  E(K)/p \lra \E(\m)/p \quad\text{defined by}\quad [P]\mapsto [\wh{P}],
\]
and it is the inverse of the inclusion map by \eqref{eq:formal-point-associated-to-P}.

If $P\in E(K)$ already lies in the formal group, it is a bit of a misnomer to call $\wh{P}$ the formal point associated to $P$ when usually $P\neq\wh{P}$. In fact, if $P=\wh{P}$ then $P$ is $(mn-1)$-torsion in $\E(\m)$, and in particular $P$ lies in $\E(\m)[p]$, which will vanish when $e<p-1$. However, the signature of $[P]$ and $[\wh{P}]$ will be preserved by Proposition \ref{prop:properties-of-signature}.

For computational purposes, the integer $mn$ may be very large and so computing $mnP$ can be slow. At the cost of multiplying by a fixed integer, we can multiply by a divisor of $n$ that varies with $P$. More precisely, given a point $P$, let $m_P\mid N$ be the order of $\ol{P}$ in $\ol{E}(k)$. Then $m_P P\in\E(\m)$ and since $m_P$ is coprime to $p$, Proposition \ref{prop:properties-of-signature} implies that
\[
  \sS([P])=\sS([\wh{P}])=\sS([m_P P]).
\]
Thus when computing signatures, we are free to choose between $\wh{P}$ and $m_P P$.

\subsection{The Somekawa $K$-group}\label{sec:somekawa}

In this section we define the Somekawa $K$-group and describe some of its properties for the case of a product of two elliptic curves with supersingular reduction. For the sake of brevity, we define it only for the product of two elliptic curves over a perfect field, though it can be done for any finite product of semiabelian varieties over any field\cite{Somekawa1990}.

\begin{definition}
  Let $K$ be a perfect field and $E_1$, $E_2$ be elliptic curves over $K$. The \emph{Somekawa} $K$-group $K(K;E_1,E_2)$ is defined as the abelian group generated by symbols of the form $\{P,Q\}_{L/K}$, where $L$ is a finite extension of $K$, and $P\in E_1(L)$ and $Q\in E_2(L)$, subject to the following types of relations:
  \begin{enumerate}
    \item ($\Z$-Bilinearity) For $P,P'\in E_1(L)$ and $Q,Q'\in E_2(L)$, then
    \begin{align*}
      \{P+P',Q\}_{L/K}=\{P,Q\}_{L/K}+\{P',Q\}_{L/K},\\
      \{P,Q+Q'\}_{L/K}=\{P,Q\}_{L/K}+\{P,Q'\}_{L/K}.
    \end{align*}
    \item (Projection Formula) Given a tower of finite extensions $L'\supseteq L\supseteq K$, with norm a restriction maps $N_{L'/L}$ and $\res_{L'/L}$, and given $P\in E_1(L)$, $P'\in E_1(L')$, $Q\in E_2(L)$, and $Q'\in E_2(L')$, then
    \begin{align*}
      \{N_{L'/L}(P'),Q\}_{L/K}=\{P',\res_{L'/L}(Q)\}_{L'/K},\\
      \{P,N_{L'/L}(Q')\}_{L/K}=\{\res_{L'/L}(P),Q'\}_{L'/K}.
    \end{align*}
    \item (Weil Reciprocity) Given a smooth complete curve $C$ defined over $K$ with morphisms $\f_1:C\ra E_1$ and $\f_2:C\ra E_2$ defined over $K$, and a function field element $f\in K(C)^\times$,
    \[
      \sum_{P\in C} \ord_P(f)\{\f_1(P),\f_2(P)\}_{K(P)/K}=0.
    \]
    We say the $C$ \emph{spans} $E_1$ and $E_2$.
  \end{enumerate}
\end{definition}

\begin{remark}
  There is clearly a surjection
  \[
    \bigoplus_{L/K\,\mathrm{finite}} E_1(L)\o_\Z E_2(L) \lra K(K;E_1,E_2)
  \]
  that maps $P\otimes Q$ to $\{P,Q\}_{L/K}$ if $P\in E_1(L)$ and $Q\in E_2(L)$.
\end{remark}

As mentioned in the introduction. The Somekawa $K$-group $K(K;E_1,E_2)$ coincides with the albanese kernel $F^2(E_1\times E_2)$. In fact, we have

\begin{theorem}\label{thm:Raskind-Spiess}(Corollary 2.4.1 of \cite{RaskindSpiess2000}, Theorem 2.4 of \cite{Somekawa1990})
  Let $E_1$ and $E_2$ be elliptic curves defined over a perfect field $K$. Then
  \[
      CH_0(E_1\times E_2)\cong\Z\oplus K(K;E_1,E_2)\oplus E_1(K)\oplus E_2(K).
  \]
\end{theorem}

\begin{remark}
  When $K$ is a $p$-adic field, the structure of $E_1(K)$ and $E_2(K)$ are well known. This is a special case of Mattuck's Theorem \cite[Theorem 7]{Mattuck1955}: if $A$ is an abelian variety of dimension $d$ defined over a $p$-adic field $K$, then $A(K)\cong\O_K^d\oplus F$ as abelian groups, where $F$ is a finite group.
\end{remark}

We are interested in the quotient group $K(K;E_1,E_2)/p$. We know that $K(K;E_1,E_2)$ (and thus $K(K;E_1,E_2)/p$) is generated by symbols of the form $\{P_1,P_2\}_{L/K}$ as $L$ ranges over all finite extensions of $K$, and these are in the images of the natural maps $P_1\otimes P_2\mapsto \{P_1,P_2\}_{L/K}$. Thus the quotient group $K(K;E_1,E_2)/p$ is generated by the images of the composite maps
\[
    \begin{tikzcd}[column sep=15pt]
        E_1(L)\o E_2(L) \arrow[r,hookrightarrow] & \bigoplus\limits_{F/K\,\text{finite}}E_1(F)\o E_2(F) \arrow[r,two heads] & K(K;E_1,E_2) \arrow[r,two heads] & K(K;E_1,E_2)/p.
    \end{tikzcd}
\]
as $L$ ranges over all finite extensions of $K$. For a fixed $L$ it is clear that the above composite map factors through
\[
    (E_1(L)\o E_2(L))/p \cong (\Z/p\Z)\o_\Z \left(E_1(L)\o_\Z E_2(L)\right)\cong E_1(L)/p\o_{\F_p} E_2(L)/p.
\]
We obtain the $\F_p$-linear map
\begin{equation}\label{eq:Phi-L-map}
    E_1(L)/p\o_{\F_p} E_2(L)/p \lra K(K;E_1,E_2)/p,\quad [P_1] \o [P_2] \mapsto \{P_1,P_2\}_{L/K}\spmod{p}.
\end{equation}
The image of this map is precisely the subgroup of $L/K$-symbols modulo $p$ which we denote by:
\[
    \Symb_{L/K}:=\text{image of}\; E_1(L)/p\o_{\F_p} E_2(L)/p \lra K(K;E_1,E_2)/p.
\]
Hence if $\Symb_{L/K}=0$ for all finite extensions $L$ of $K$, then $K(K;E_1,E_2)/p=0$ which in turn proves Conjecture \ref{conj:CT} for $E_1\times E_2$. This is very difficult, so we address the ``first level'' of this problem. Namely, the question we study in this paper is the following.

\begin{letterconjecture}
  Let $K$ be a $p$-adic field, $E_1$ and $E_2$ elliptic curves over $K$ with supersingular reduction. Then $\Symb_{K/K}=0$.
\end{letterconjecture}

This takes us to considering the structure of $E_1(K)/p\o E_2(K)/p$. When the curves $E_1$ and $E_2$ have supersingular reduction, then the structure is completely determined by the formal groups by Lemma \ref{lemma:formal-group-mod-p} which tells us that the inclusion map $\E_i(\m)\hookrightarrow E_i(K)$ induces an isomorphism
\[
  \E_i(\m)/p \xrightarrow{\sim} E_i(K)/p
\]
with inverse $[P]\mapsto[\wh{P}]$, where $\wh{P}$ is the formal point associated to $P$.

Furthermore, if the absolute ramification index $e$ of $K$ is small, i.e. $e<p-1$, then Theorem \ref{thm:image_of_mult_by_p_map_formal_groups}, or more precisely Remark \ref{rm:Fp-vector-space-structure-of-Emodp}, gives us a set
\[
  \{\b_{i,j}\o\gamma_{i',j'}\in \E_1(\m)\o\E_2(\m)\mid 1\leq i,i'\leq e,1\leq j,j'\leq f\}
\]
whose classes modulo $p$ form an $\F_p$-basis of $\E_1(\m)/p\o\E_2(\m)/p$. Thus if
\begin{equation}\label{eq:brute-force-relations}
  \{\b_{i,j}\o\gamma_{i',j'}\}_{K/K}\spmod{p}=0,
\end{equation}
for all $i,j,i',j'$, then $\Symb_{K/K}=0$.

The above condition can be difficult to guarantee since we would need to produce $[K:\Q_p]^2$ relations in $K(K;E_1,E_2)/p$. The above basis is stratified by the valuation $\nu$ on $K$ since $\nu(\b_{i,j})=\nu(\gamma_{i,j})=i$, and so we can approach the problem of producing relations of the form \eqref{eq:brute-force-relations} ``valuation-by-valuation''. The pair of valuations $(\nu(\wh{P}_1),\nu(\wh{P}_2))$, will be essential for comparing the symbol $\{P_1,P_2\}_{K/K}$ to the above basis. This motivates the following definition.

\begin{definition}
  Let $K$ be a $p$-adic field and let $E_1,E_2$ be elliptic curves over $K$ with supersingular reduction. Let $\{P_1,P_2\}_{K/K}$ be a symbol in $K(K;E_1,E_2)$, where $P_\ell\in E_\ell(L)$ for $\ell=1,2$. The \emph{signature} of $\{P_1,P_2\}_{K/K}$ is defined as the pair of integers
  \[
    \sS(\{P_1,P_2\}_{K/K})=\left(\nu(\wh{P}_1),\nu(\wh{P}_2)\right)
  \]
  where $\wh{P}_\ell\in\E_\ell(\m)$ is the associated formal point of $P_\ell$ (see Definition \ref{def:associated-formal-point}).
\end{definition}

\begin{remark}
  Using the signature of points in $\E_\ell(\m)/p$ we defined in Section \ref{sec:formal-groups} we have:
\[
  \sS(\{P_1,P_2\}_{K/K})=\big(\sS([P_1]),\sS([P_2])\big).
\]
\end{remark}

Our proposed method of showing that $K/K$-symbols vanish in $K(K;E_1,E_2)/p$ can be roughly described as ``producing enough relations of the correct signature''. We first do some low degree examples to illustrate this method and then make it precise in Proposition \ref{prop:necessary-conditions-for-PhiL-0} below.

\begin{example}\label{ex:1-1-signature}($K=\Q_p$)
  Since $[K:\Q_p]=1$, then $\E_1(\m)/p\o\E_2(\m)/p$ is 1-dimensional over $\F_p$. Now suppose that we've produced a symbol $\{P_1,P_2\}_{K/K}$ with signature $(1,1)$ such that $\{P_1,P_2\}_{K/K}\equiv0\spmod{p}$. Since the signature is $(1,1)$, $[\wh{P}_1]\o[\wh{P}_2]$ generates $\E_1(\m)/p\o\E_2(\m)/p$ and thus $[P_1]\o[P_2]$ generates $E_1(K)/p\o E_2(K)/p$. Since
  \[
    [P_1]\o[P_2]\mapsto\{P_1,P_2\}_{L/K}\spmod{p}=0,
  \]
  then $\Symb_{K/K}=0$.
\end{example}

\begin{example}\label{ex:quadratic-ramified}($K=$ quadratic ramified extension of $\Q_p$, $p>3$) In this case $e=2$ and $f=1$ so $\E_1(\m)/p\o\E_2(\m)/p$ is 4-dimensional. Now suppose that we have four symbols
  \[
    \{P_1,P_2\}_{K/K},\{Q_1,Q_2\}_{K/K},\{R_1,R_2\}_{K/K},\{S_1,S_2\}_{K/K}\in K(K;E_1,E_2)
  \]
  of signatures $(1,1),(1,2),(2,1)$ and $(2,2)$ respectively, such that they are zero modulo $p$, e.g. $\{P_1,P_2\}_{K/K}\equiv 0\spmod{p}$. The condition on their signatures, together with Remark \ref{rm:Fp-vector-space-structure-of-Emodp}, implies that
  \[
    \{[\wh{P}_1]\o[\wh{P}_2],[\wh{Q}_1]\o[\wh{Q}_2],[\wh{R}_1]\o[\wh{R}_2],[\wh{S}_1]\o[\wh{S}_2]\}
    \]
  generates $\E_1(\m)/p\o\E_2(\m)/p$. Thus $\Symb_{K/K}=0$.
\end{example}

\begin{proposition}\label{prop:necessary-conditions-for-PhiL-0}
  Let $K$ be a $p$-adic field with absolute ramification index $e<p-1$. Let $E_1,E_2$ be elliptic curves over $K$ with supersingular reduction. Suppose we have a set of $K/K$-symbols
  \[
    \left\{\{P_{ij},Q_{i'j'}\}_{K/K}\in K(K;E_1,E_2)\mid 1\leq i,i'\leq e,1\leq j,j'\leq f\right\}
  \]
  such that
  \begin{enumerate}[label=(\roman*)]
    \item\label{item:1-Phi_L} $\{P_{ij},Q_{i'j'}\}_{K/K}\equiv 0\pmod{p}$ for all $i,i',j,'$.
    \item\label{item:2-Phi_L} $\sS(\{P_{ij},Q_{i'j'}\}_{K/K})=(i,i')$,
    \item\label{item:3-Phi_L} Both $\{\wh{P}_{i1}+\E_1(\m^{i+1}),\ldots,\wh{P}_{if}+\E_1(\m^{i+1})\}$  and $\{\wh{Q}_{i'1}+\E_2(\m^{i+1}),\ldots,\wh{Q}_{i'f}+\E_1(\m^{i+1})\}$ are $\F_p$-basis of $k\cong\E_\ell(\m^i)/\E_\ell(\m^{i+1})$ ($\ell=1,2$) for all $i$ and $i'$.
  \end{enumerate}
  Then $\Symb_{K/K}=0$.
\end{proposition}

\begin{proof}
  Conditions \ref{item:2-Phi_L} and \ref{item:3-Phi_L} imply via Remark \ref{rm:Fp-vector-space-structure-of-Emodp} that
  \[
    \{[\wh{P}_{ij}]\o[\wh{Q}_{i'j'}]\mid 1\leq i,i'\leq e,1\leq j,j'\leq f\}
  \]
  is an $\F_p$-basis of $\E_1(\m)/p\o \E_2(\m)/p$ and hence
  \[
    \{[P_{ij}]\o[Q_{i'j'}]\mid 1\leq i,i'\leq e,1\leq j,j'\leq f\}
  \]
  is a basis of $E_1(K)/p\o E_2(K)/p$. Finally, condition \ref{item:1-Phi_L} implies that each basis element
  \[
    [P_{ij}]\o[Q_{i'j'}]\mapsto\{P_{ij},Q_{i'j'}\}_{L/K}\equiv 0\pmod{p},
  \]
  under the map \eqref{eq:Phi-L-map}, and thus its image $\Symb_{K/K}$ is zero.
\end{proof}

\section{Main results}\label{sec:scholten_curves}


\subsection{Weil Reciprocity}\label{sec:weil-reciprocity}

In this section we produce Weil Reciprocity relations in $K(K;E_1,E_2)$ by using a particular family of hyperelliptic curves due to J. Scholten in an unpublished paper and generalized by Gazaki and Love \cite[\S3]{GazakiLove2023}. First we define these curves and show that there always exists a suitable one for computing the signature of the resulting relations.

An elliptic curve with full 2-torsion over $K$ can be written as
\[
  E_{a,b}:y^2=x(x-a)(x-b),
\]
where $a,b\neq 0$ and $a\neq b$. Given a pair $E_{a,b}$ and $E_{c,d}$ of curves with full 2-torsion, Scholten defined a curve that mapped nontrivially into both $E_{a,b}$ and $E_{c,d}$.

\begin{definition}\label{def:scholten-curve}
  Let $a,b,c,d\in K^\times$ with $a\neq b$ and $c\neq d$. Then we define the \emph{Scholten} curve $C_{a,b,c,d}$ by the affine equation
  \[
    (ad-bc)Y^2 = ((a-b)X^2-(c-d))(aX^2-c)(bX^2-d).
  \]
  If $ad-bc\neq0$, there are maps $\f_{a,b}:C_{a,b,c,d}\ra E_{a,b}$ and $\f_{c,d}:C_{a,b,c,d}\ra E_{c,d}$ defined over $K$:
  \begin{equation}\label{eq:definitions-of-fiab-ficd}
    \begin{gathered}
      \f_{a,b}(X,Y)=\left(\frac{ab(a-b)}{ad-bc}\left(X^2-\frac{c-d}{a-b}\right),\frac{ab(a-b)}{ad-bc}Y\right),\\
      \f_{c,d}(X,Y)=\left(\frac{cd(a-b)}{(ad-bc)X^2}\left(X^2-\frac{c-d}{a-b}\right),-\frac{cd(c-d)}{(ad-bc)X^3}Y\right).
    \end{gathered}
    \end{equation}
\end{definition}

\begin{notation*}
  If there is a curve $C$ and maps $C\ra E_1$, $C\ra E_2$, all defined over $K$, we say that $E_1\leftarrow C\rightarrow E_2$ is a \emph{span} of $E_1$ and $E_2$ over $K$, or that $C$ \emph{spans} $E_1$ and $E_2$ over $K$.
\end{notation*}

\begin{remark}\label{rm:C-is-smooth}
  The right hand side of the defining equation for $C:=C_{a,b,c,d}$ is a product of three quadratics, and any pair of these have the same resultant which is equal to $(ad-bc)^2$. Thus if $ad-bc\neq 0$, then $C$ is a smooth, geometrically irreducible curve of genus 2 and comes equipped with the usual hyperelliptic involution
\[
    \iota : C \lra C \quad\text{defined by}\quad \iota(x,y)=(x,-y).
\]
Points fixed under this involution are called \emph{Weierstrass points}.
\end{remark}

The fact that $C$ is hyperelliptic will allow us to produce a rational function on $C$ quite easily and ultimately produce Weil Reciprocity relations.

\begin{proposition}\label{prop:symbols_coming_from_scholten_curves}
  Let $K$ be a field, $E_1$ and $E_2$ elliptic curves over $K$. Suppose there is a smooth hyperelliptic curve $C$ that spans $E_1$ and $E_2$ over $K$, via maps $\s_i:C\ra E_i$ that satisfy $\s_i(\iota(P))=-\s_i(P)$ for every point $P\in C(K)$. If $P$ is not fixed by the hyperelliptic involution, we have
  \[
    4\{\s_1(P),\s_2(P)\}_{K/K} = 0
  \]
  in $K(K;E_1,E_2)$.
\end{proposition}

\begin{proof}
  First let us suppose that $C$ has a $K$-rational Weierstrass point $W$. The $x$-coordinate map $x:C\ra\P^1$ is a degree 2 rational map on $C$ and clearly $x = x\circ\iota$. The pullback of the divisor $D=(x(P))-(x(W))\in\mathrm{Pic}(\P^1)$ along $x:C\ra\P^1$ is
  \[
    x^*(D) = (P) + (\iota P) - 2 (W).
  \]
  Since $\mathrm{Pic}(\P^1)=\Z$ via the degree map, then $D$ is principal, say $D=\div(g)$ for $g\in K(\P^1)$. Thus
  \[
    (P) + (\iota P) - 2 (W) = x^*(D) = x^*(\div(g))=\div(x^* g).
  \]
  That is, $x^*(D)$ is principal.
  
  If we write $f = x^* g$ then the Weil relation attached to $C$, $\s_1,\s_2$ and $f$ is
  \[
    \{\s_1(P),\s_2(P)\}_{K/K}+\{\s_1(\iota(P)),\s_2(\iota(P))\}_{K/K} - 2\{\s_1(W),\s_2(W)\}_{K/K}.
  \]
  By the second assumption, have $\s_i(\iota P)=-\s_i(P)$ and in particular $\s_i(W)$ is 2-torsion. Using the bilinearity of $K/K$-symbols one concludes that
  \[
    0=2\{\s_1(P),\s_2(P)\}_{K/K}.
  \]
  Multiplying by 2 yields the result.

  If $C$ contains no $K$-rational Weierstrass points, then any Weierstrass point $W$ will be defined over a quadratic extension $L$ of $K$. In which case the principal divisor we have is $(P)+(\iota P)-(W)$, which yields the WR relation
  \[
    2\{\s_1(P),\s_2(P)\}_{K/K}-\{\s_1(W),\s_2(W)\}_{L/K}.
  \]
  Multiplying by 2 yields
  \[
    0=4\{\s_1(P),\s_2(P)\}_{K/K}-2\{\s_1(W),\s_2(W)\}_{L/K}=4\{\s_1(P),\s_2(P)\}_{K/K}.
  \]
\end{proof}

Given two elliptic curves $E_{a,b}$ and $E_{c,d}$ with full 2-torsion, we want to apply Proposition \ref{prop:symbols_coming_from_scholten_curves} to the Scholten curve $C_{a,b,c,d}$, but it may be that $ad-bc=0$ or that $C_{a,b,c,d}$ may not have a $K$-rational Weierstrass point. However, the construction of $C_{a,b,c,d}$ is flexible enough to produce more than just the one span from Definition \ref{def:scholten-curve}. Furthermore, if both elliptic curves have supersingular reduction, then we can always construct a Scholten curve and span that satisfies the hypothesis of Proposition \ref{prop:symbols_coming_from_scholten_curves}. To check this, we verify that such elliptic curves over totally ramified fields can be written in Legendre form with certain conditions on its $j$-invariant.

\begin{lemma}\label{lemma:p-3mod4-and-legendre-form}
  Suppose $K$ is a totally ramified extension of $\Q_p$ with $p>3$. Let $E$ be an elliptic curve will full 2-torsion and supersingular reduction. Then
  \begin{enumerate}[label=(\roman*)]
    \item $p\equiv3\pmod{4}$, in particular $-1$ is not a square in $K^\times$.
    \item There exists $\l\in\O$ such that $\l\not\equiv 0,1\spmod{\m}$ and such that $E$ is isomorphic over $K$ to a Legendre form elliptic curve
    \[
      E_\l:y^2=x(x-1)(x-\l).
    \]
    \item The $j$-invariant $j(E)$ satisfies $j(E)\not\equiv 0\spmod{\m}$.
  \end{enumerate}
\end{lemma}
\begin{proof}$\;$

  \begin{enumerate}[label=(\roman*)]
    \item Since $E$ has supersingular reduction and the residue field $k$ is $\F_p$ by the assumption on the ramification, we have that $|\ol{E}(k)|=p+1$ by Remark \ref{rm:about-ssing-red}. Since $E$ also has full 2-torsion, then $\ol{E}(k)[2]$ is an order 4 subgroup of $\ol{E}(k)$ and thus Lagrange's theorem tells us that $4\mid p+1$.
    
    \item Let $y^2=x^3+a_2x^2+a_4x+a_6$ be a minimal Weierstrass equation for $E$. Since $E$ has full 2-torsion, then the cubic in $x$ is of the form $(x-e_1)(x-e_2)(x-e_3)$ with $e_j\in K$. In fact, $e_i\in\O$ since they are roots of a monic polynomial with coefficients in the integrally closed ring $\O$. Thus the discriminant $\Delta_E$ of $E$ is
    \[
      \Delta_E = 16(e_1-e_2)^2(e_1-e_2)^2(e_2-e_3)^2.
    \]
    Since the Weierstrass equation is minimal and $E$ has good reduction, we have $\nu(\Delta_E)=0$, and since each $e_i$ is integral, then $\nu(e_i-e_j)\geq0$ so that $\nu(e_i-e_j)=0$ for $i\neq j$. That is $e_1-e_2,e_1-e_3$ and $e_2-e_3$ are all units.
    
    The first part tells us that $-1$ is not a square in $k=\F_p$ so either $e_1-e_2$ or $e_2-e_1$ is a square in $K$. Thus if we make the substitution $x=(e_1-e_2)X+e_2,y=(e_1-e_2)^{3/2}Y$, respectively $x=(e_2-e_1)X+e_1,y=(e_2-e_1)^{3/2}Y$, we obtain the Legendre form elliptic curve
    \[
      E_\l : Y^2 = X(X-1)(X-\l),\quad \l=\frac{e_3-e_2}{e_1-e_2},\;\text{resp.}\;\frac{e_3-e_1}{e_2-e_1}.
    \]
    Notice that the change of variable that was made is of the form $x=u^2X+r$, $y=u^3Y$ with $u$ a unit. Thus the discriminant of $E_\l$ satisfies:
    \[
      0=\nu(\Delta_E)=\nu(u^{12}\Delta_{E_\l})=\nu(\Delta_{E_\l})=\nu(16\l^2(1-\l)^2)=2\nu(\l)+2\nu(1-\l).
    \]
    Since $\l$ is integral, then we conclude that $\l\not\equiv0,1\spmod{\m}$ as required.

    \item Suppose that $j(E)\equiv0\spmod{\m}$. By the previous item, we have that $j(E)=j(E_\l)=2^8(\l^2-\l+1)^3\l^{-2}(1-\l)^{-2}$ for some unit $\l$. Since $\l\not\equiv 0,1\spmod{\m}$, then $\l^2-\l+1\equiv0\spmod{\m}$. Hence $\bar{\l}\in k$ is a solution to the equation $t^2-t+1$ in $k=\F_p$. Since $t^2-t+1\mid t^3+1$ and $\bar{\l}\neq -1$, then $\bar{\l}$ is a 6th root of unity in $\F_p^\times$. This implies that $6\mid p-1$ and in particular $p\equiv1\spmod{3}$. However, it is well known that for $p\equiv1\spmod{3}$ there are no elliptic curves over $\F_p$ with $j$-invariant 0 \cite[V.4.4]{Silverman2013}. Thus $j(E)\not\equiv 0\spmod{\m}$.
  \end{enumerate}
\end{proof}

The above lemma tells us that if we have elliptic curves $E_1$ and $E_2$ over a totally ramified field $K$ with full 2-torsion and supersingular reduction, then we may assume that $E_1=E_\l$ and $E_2=E_\mu$ for some $\l,\mu\in\O$. The Legendre form of an elliptic curve is not unique up to isomorphism, so choosing a different form will produce a different genus 2 curve.

Observe that for a general elliptic curve $E_{a,b}$ with full 2-torsion, we have that
\[
  E_{a,b}=E_{b,a}\cong E_{-a,b-a}=E_{b-a,-a}\cong E_{a-b,-b}=E_{-b,a-b},
\]
where the first isomorphism is $(x,y)\mapsto(x+a,y)$ and the second is $(x,y)\mapsto (x+b,y)$. Thus, given two elliptic curves $E_{a,b}$ and $E_{c,d}$, the six scholten curves
\begin{equation}\label{eq:six-hyperelliptic-curves}
  C_{a,b,c,d},C_{b,a,c,d},C_{-a,b-a,c,d},C_{b-a,-a,c,d},C_{a-b,-b,c,d},C_{-b,a-b,c,d}
\end{equation}
all span isomorphic elliptic curves. It is worth pointing out that the six hyperelliptic curves in \eqref{eq:six-hyperelliptic-curves} are pairwise non-isomorphic \cite[\S4.2.2]{GazakiLove2024}.

So if we postcompose these spans with the above isomorphisms we obtain six spans from $E_{a,b}$ to $E_{c,d}$:
\[
  \begin{tikzcd}
    E_{a,b} & C \arrow[l,"\sigma_1"'] \arrow[r,"\sigma_2"] & E_{c,d}
  \end{tikzcd}
\]
The formulas for the six Scholten curves and their associated maps are easily computed. In table \ref{table:scholten_maps-condensed} we list these formulas for the case when the elliptic curves are given in Legendre form: $E_{a,b}=E_{1,\l}$ and $E_{c,d}=E_{1,\mu}$. We will see that at least one of these spans will produce a Weil Reciprocity relation $\{P_1,P_2\}_{K/K}=0$ with diagonal signature.

\begin{table}
  \caption{Formulas for the six spans $E_{1,\l}\xleftarrow{\s_1}C\xrightarrow{\s_2}E_{1,\mu}$ where $C$ runs through the six hyperelliptic curves of \eqref{eq:six-hyperelliptic-curves}.}
  \label{table:scholten_maps-condensed}
  \begin{center}
    {\renewcommand{\arraystretch}{1}
    \begin{tabular}{|c|c|c|c|c|c|c|}
      \hline
      \multicolumn{7}{|c|}{\rule{0pt}{8ex} $\begin{gathered} C:y^2 = u(x^2-r)(x^2-s)(x^2-t)\\ \sigma_1:C\ra E_{1,\l},\quad\sigma_1(x,y)=\left(u(x^2-v),uy\right)\\ \sigma_2:C\ra E_{1,\mu},\quad\sigma_2(x,y)=\left(ust(x^2-r)x^{-2},-urstx^{-3}y\right)\end{gathered}$} \\[4ex]
      \hline
      $C$ & $\delta=ad-bc$ & $u$ & $r$ & $s$ & $t$ & $v$\\
      \hline
      $C_1=C_{1,\l,1,\mu}$ & $\mu-\l$ & $\dfrac{\l(1-\l)}{\mu-\l}$ & $\dfrac{1-\mu}{1-\l}$ & $1$ & $\dfrac{\mu}{\l}$ & $\dfrac{1-\mu}{1-\l}$ \\[2ex]
      \hline
      $C_2=C_{\l,1,1,\mu}$ & $1-\l\mu$ & $\dfrac{\l(1-\l)}{1-\l\mu}$ & $-\dfrac{1-\mu}{1-\l}$ & $\dfrac{1}{\l}$ & $\mu$ & $-\dfrac{1-\mu}{1-\l}$ \\[2ex]
      \hline
      $C_3=C_{-1,\l-1,1,\mu}$ & $\mu-1+\l$ & $\dfrac{\l(1-\l)}{\mu-1+\l}$ & $-\dfrac{1-\mu}{\l}$ & $-\dfrac{\mu}{1-\l}$ & $-1$ & $-\dfrac{\mu}{1-\l}$\\[2ex]
      \hline
      $C_4=C_{\l-1,-1,1,\mu}$ & $1-\mu+\l\mu$ & $\dfrac{\l(1-\l)}{1-\mu+\l\mu}$ & $\dfrac{1-\mu}{\l}$ & $-\dfrac{1}{1-\l}$ & $-\mu$ & $-\dfrac{1}{1-\l}$ \\[2ex]
      \hline
      $C_5=C_{1-\l,-\l,1,\mu}$ & $\l\mu-\mu-\l$ & $\dfrac{\l(1-\l)}{\l\mu-\mu-\l}$ & $1-\mu$ & $\dfrac{1}{1-\l}$ & $-\dfrac{\mu}{\l}$ & $\dfrac{1}{1-\l}$ \\[2ex]
      \hline
      $C_6=C_{-\l,1-\l,1,\mu}$ & $\l-1-\l\mu$ & $\dfrac{\l(1-\l)}{\l-1-\l\mu}$ & $\mu-1$ & $\dfrac{\mu}{1-\l}$ & $-\dfrac{1}{\l}$ & $\dfrac{\mu}{1-\l}$\\[2ex]
      \hline
    \end{tabular}
    }
  \end{center}
\end{table}

\begin{lemma}\label{lemma:existence-of-good-span}
  Let $K/\Q_p$ be a totally ramified extension with $p>3$. Let $E_1$ and $E_2$ be elliptic curves over $K$ with full 2-torsion and supersingular reduction. Then there exists a smooth genus 2 curve $C$ that spans $E_1$ and $E_2$. The curve $C$ has defining equation of the form
  \[
    C:Y^2=u(X^2-r)(X^2-s)(X^2-t),
  \]
  where $u,r,s,t\in\O^\times$, and the spanning maps $\s_1:C\ra E_1$ and $\s_2:C\ra E_2$ have the following form
  \[
    \s_1(X,Y)=(u(X^2-v),uY),\quad \s_2(X,Y)=(ust(X^2-r)X^{-2},-urst X^{-3}Y),
  \]
  where $v$ is a unit. Furthermore, $C$ has a $K$-rational Weierstrass point $W=(w,0)$ where $w\in \O^\times$.
\end{lemma}

\begin{proof}
  By Lemma \ref{lemma:p-3mod4-and-legendre-form}, we may write $E_1=E_\l$ and $E_2=E_\mu$ for some $\l,\mu\in \O$ such that $\l,\mu,1-\l$ and $1-\mu$ are all units. Furthermore, $p\equiv3\spmod{4}$ so that $-1$ is not a square in $K^\times$. The curve we are looking for will be taken from Table \ref{table:scholten_maps-condensed}. For simplicity, we label all the curves and constants in the table with a subindex $i$ that indicates the row of the table. For example $C_2$ will have the equation $y^2=u_2(x^2-r_2)(x^2-s_2)(x^2-t_2)$, etc.
  
  There are several cases in this proof, but they all follow the same structure. We will show that a curve $C_i$ in Table \ref{table:scholten_maps-condensed} is smooth by showing that $\delta_i=ad-bc\not\equiv 0\spmod{\m}$ (see Remark \ref{rm:C-is-smooth}). Next, we show that either $r_i$ or $s_i$ is a square in $K^\times$, which implies that the corresponding smooth curve will have a $K$-rational point of the form $(w,0)$ where $w=\sqrt{r_i}$ or $w=\sqrt{s_i}$. All the $r$, $s$ and $t$ in the table are units since they are ratios of the units $\l,\mu,1-\l$ and $1-\mu$. Thus $w$ is also a unit.
  
  We proceed to the proof. First we split into two cases: $\l\not\equiv\mu\spmod{\m}$ and $\l\equiv\mu\spmod{\m}$. We drop the ``$\spmod{\m}$'' from the notation for the rest of the proof.

  \begin{enumerate}[label=\arabic*.]
    \item Suppose $\l\not\equiv\mu$. In this case, $\d_1\equiv\mu-\l\not\equiv 0$ for $C_1$ so it is smooth. If $s_1=(1-\mu)/(1-\l)$ is a square, then $C_1$ works. If $s_1$ is not a square, then $s_2=-s_1=-(1-\mu)/(1-\l)$ is a square, say $\eta^2$.
    \begin{enumerate}[label*=\arabic*.]
      \item If $\d_2=1-\mu\l\not\equiv0$, then $C_2$ is smooth and has the Weierstrass point $(\eta,0)$.
      \item If $\d_2\equiv0$, then $\mu\equiv1/\l$ and thus $\eta^2\equiv-(1-\l^{-1})/(1-\l)$. Expanding this congruence out implies that $\l$ satisfies the quadratic equation $0\equiv(\l-1)(\l-\eta^{-2})$. Since $\l\not\equiv1$, then $\l\equiv\eta^{-2}$ and $\mu\equiv\eta^2$ are both squares. Next observe that
      \begin{gather*}
        0\equiv\d_3=\mu-1+\l \quad\Lra\quad 0\equiv\l\mu-\l+\l^2\equiv1-\l+\l^2,\\
        0\equiv\d_5=\l\mu-\mu-\l\quad\Lra\quad 0\equiv\l^2\mu-\l\mu-\l^2\equiv\l-1-\l^2,
      \end{gather*}
      Notice that $1-\l+\l^2\not\equiv 0$ since otherwise $j(E_\l)=2^8(\l^2-\l+1)^3\l^{-2}(1-\l)^{-2}\equiv0$ which is not the case by Lemma \ref{lemma:p-3mod4-and-legendre-form}. Thus necessarily $\d_3,\d_5\not\equiv0$, that is $C_3$ and $C_5$ are smooth. If $1-\l$ is a square, then so is $s_5=(1-\l)^{-1}$, and thus $C_5$ has a Weierstrass point. If $1-\l$ is not a square, then $s_3=-\mu/(1-\l)$ is a square since $\mu\equiv\eta^2$ is a square. Thus $C_3$ has a Weierstrass point.
    \end{enumerate}
    \item Suppose $\l\equiv\mu$.
    \begin{enumerate}[label*=\arabic*.]
      \item If $\d_3=\mu-1+\l\equiv0$, then $2\l\equiv1$. First, we have that $C_4$ is smooth since otherwise we would have
      \[
        0\equiv\d_4=1-\mu+\l\mu\equiv 1-\l+\l^2 \quad\Lra\quad j(E_\l)\equiv0,
      \]
      which contradicts Lemma \ref{lemma:p-3mod4-and-legendre-form}. Next we also have that $C_5$ is smooth. Indeed, if
      \[
        0\equiv\d_5=\l\mu-\mu-\l\equiv \l^2-2\l\equiv\l^2-1,
      \]
      then $\l\equiv-1$ since $\l\not\equiv1$ by Lemma \ref{lemma:p-3mod4-and-legendre-form}. However, this would imply that $1\equiv2\l\equiv-2$ which contradicts the fact that $p\neq 3$. Now, if $s_5=1/(1-\l)$ is a square, then $C_5$ has a Weierstrass point. Otherwise, $s_4=-s_5$ is a square so that $C_4$ works.
      \item If $\d_3\not\equiv0$, then $C_3$ is smooth. Furthermore, if $s_3=-\mu/(1-\l)$ is a square, then $C_3$ has a Weierstrass point. If $s_3$ is not a square, then $s_6=-s_3$ is a square. Observe that $\d_6=\l-1-\l\mu\equiv-(\l^2-\l+1)\not\equiv 0$ since $j(E_\l)\not\equiv0$. Thus $C_6$ is smooth and has a Weierstrass point.
    \end{enumerate}
  \end{enumerate}
\end{proof}

\begin{remark}\label{rm:defining-poly-is-separable}
  As a consequence of the proof of Lemma \ref{lemma:existence-of-good-span}, we see that the degree six polynomial $F(X)\in\O[X]$ has separable reduction. More precisely, if the curve $C$ given to us by Lemma \ref{lemma:existence-of-good-span} is defined by $Y^2=F(X)$, with $F\in\O[X]$, then the reduced polynomial $\ol{F}\in k[X]$ is separable since the resultant of any two of its quadratic factors is a unit by Remark \ref{rm:C-is-smooth}.
\end{remark}

\subsection{Diagonal Signatures}\label{sec:diagonal-symb}

We have shown in Lemma \ref{lemma:existence-of-good-span} that given a pair of elliptic curves $E_1$ and $E_2$ with full 2-torsion and supersingular reduction, we can produce a span $E_1\leftarrow C\rightarrow E_2$ to which to we can apply Proposition \ref{prop:symbols_coming_from_scholten_curves}. Indeed, one immediately sees from the $y$-coordinate of $\s_1(x,y)$ and $\s_2(x,y)$, and the formula for the hyperelliptic involution, that $\s_i(P)=-\s_i(\iota P)$.

This produces $K/K$-symbols $\{P_1,P_2\}_{K/K}$ in $K(K;E_1,E_2)$. The last step is to compute the signatures of these symbols. This is done by multiplying the entries $P_1$ and $P_2$ by integers (cf. Definition \ref{def:associated-formal-point}).

The next lemma gives us a way to determine how deep into the filtration a point $P_i$ goes when multiplied by a certain integer. Roughly speaking, if a $K$-rational point $P$ on any elliptic curve $E$ is $\pi$-adically close to an $n$-torsion point, with $n$ coprime to $p$, then $nP$ will be $\pi$-adically close to the identity, i.e. it will lie in the formal group. Furthermore, the closer the point is to the torsion point, the deeper $nP$ will be in the filtration of $\wh{E}(\m)$.

\begin{lemma}\label{lemma:key-lemma}
  Let $E$ be an elliptic curve defined over $K$ with good reduction, $n$ coprime to $p$, and $N>0$. Let $P\in E(K)$. Then the following are equivalent.
  \begin{enumerate}[label=(\roman*)]
    \item\label{item:nP-in-Ehat-N} $nP\in\E(\m^N)\setminus\E(\m^{N+1})$,
    \item\label{item:congruent-to-n-torsion-mod-Ehat-N} There exists $T\in E(K)[n]$ such that $P-T\in\E(\m^N)\setminus\E(\m^{N+1})$.
  \end{enumerate}
\end{lemma}

\begin{proof}
  First we show that
  \begin{equation}\label{eq:prelim-implication}
    nP\in\E(\m^N) \quad\iff\quad \text{there exists}\; T\in E(K)[n]\;\text{such that}\; P-T\in\E(\m^N),
  \end{equation}
  and then deduce the full Lemma. The forward implication of \eqref{eq:prelim-implication} is immediate since $n(P-T)=nP$. For the other implication, first we prove it when $N=1$, then derive the case $N>1$.

  Consider the following commutative diagram
  \[
    \begin{tikzcd}
      0 \arrow[r] & \E(\m) \arrow[r] \arrow[d,"\lbrack n\rbrack"] & E(K) \arrow[r] \arrow[d,"\lbrack n\rbrack"]  & \ol{E}(k) \arrow[r] \arrow[d,"\lbrack n\rbrack"]  & 0 \\
      0 \arrow[r] & \E(\m) \arrow[r] & E(K) \arrow[r] & \ol{E}(k) \arrow[r] & 0 
    \end{tikzcd}
  \]
  Since multiplication-by-$n$ in $\E(\m)$ is an isomorphism (since $n$ is coprime to $p$), then the Snake Lemma tells us that the reduction map induces an isomorphism $E(K)[n]\xrightarrow{\sim}\ol{E}(k)[n]$. Now, since $nP\in\E(\m)$, then $\ol{nP}=O\in\ol{E}(k)$ and in particular $\ol{P}\in\ol{E}(k)[n]$. Thus the above isomorphism tells us that there exists $T\in E(K)[n]$ that reduces to $\ol{P}$, i.e. $P-T\in\E(\m)$.

  Now suppose $N>1$. Since $\E(\m^N)\subseteq\E(\m)$, the $N=1$ case gives us a point $T'\in E(K)[n]$ such that $Q:=P-T\in \E(\m)$. In particular $nQ=n(P-T)=nP\in\E(\m^N)$. Since $n\in\O^\times$, then $\nu(Q)=\nu(nQ)\geq N$ by Proposition \ref{prop:properties-of-signature}, and therefore $P-T\in\E(\m^N)$ as required. This proves \eqref{eq:prelim-implication}.

  Now, we show \ref{item:nP-in-Ehat-N}$\Lra$\ref{item:congruent-to-n-torsion-mod-Ehat-N}. By \eqref{eq:prelim-implication} there is a $T\in E(K)[n]$ such that $P-T\in\E(\m^{N})$. If $P-T\in\E(\m^{N+1})$, then \eqref{eq:prelim-implication} would imply that $nP\in\E(\m^{N+1})$, which is not true. Thus $P-T\not\in\E(\m^{N+1})$ as required.

  For the reverse implication, suppose there exists $T\in E(K)[n]$ such that $P-T\in\E(\m^N)\setminus\E(\m^{N+1})$. By \eqref{eq:prelim-implication}, $nP\in\E(\m^N)$. If $nP\in\E(\m^{N+1})$, then \eqref{eq:prelim-implication} would imply that there exists $T'\in E(K)[n]$ such that $P-T'=(P-T)-(T'-T)\in\E(\m^{N+1})$ and hence \eqref{eq:prelim-implication} applied to $P-T'$ would imply that $n(P-T')=nP\in\E(\m^{N+1})$, a contradiction. Thus $nP\not\in\E(\m^{N+1})$ as required.
\end{proof}

We are now in a position to prove Theorem \ref{thm:main-theorem-intro} from the introduction.

\begin{theorem}\label{thm:main-theorem}
  Let $K$ be a totally ramified extension of $\Q_p$, with $p>3$. Let $E_1$ and $E_2$ be elliptic curves with full 2-torsion and supersingular reduction. Then for every positive integer $N$, there exist points $P_N\in \E_1(K)$ and $Q_N\in \E_2(K)$ such that the $K/K$-symbol $\{P_N,Q_N\}_{K/K}\in K(K;E_1,E_2)$ has signature $(N,N)$ and such that
  \[
    \{P_N,Q_N\}_{K/K}\equiv 0\pmod{p}.
  \]
\end{theorem}

\begin{proof}
  By Lemma \ref{lemma:existence-of-good-span} there is a smooth genus 2 curve $C$ that spans $E_1$ and $E_2$. The defining equation of $C$ is of the form $y^2=F(x)$ where
  \[
    F(X)=u\left(X^2 - r\right)\left(X^2-s\right)\left(X^2-t\right)\in\O[X].
  \]
  The spanning maps $\s_1:C\ra E_1$ and $\s_2:C\ra E_2$ are of the form
  \[
    \s_1(X,y)=(u(X^2-v),uY),\quad \s_2(X,Y)=(ust(X^2-r)X^{-2},-urst X^{-3}Y),
  \]
  where all the constants $u,r,s,t$ and $v$ are units. Finally, Lemma \ref{lemma:existence-of-good-span} tells us that $C$ has a $K$-rational Weierstrass point $W=(w,0)$ with $w\in\O^\times$.

  Since $\s_i(W)$ is 2-torsion in $E_i$, Lemma \ref{lemma:key-lemma} suggests to find a point $P\in C(K)$ that is $\pi$-adically close to $W$. We do this with Hensel's Lemma. For the given $N>0$, consider the polynomial
  \[
    f(X):=\pi^{2N}-F(X)\in\O[X].
  \]
  Notice that $f(w)=\pi^{2N}-F(w)=\pi^{2N}\equiv0\spmod{\pi^{2N}}$ since $W=(w,0)$ satisfies the equation of $C$. That is, $w$ is an approximate root of $f(w)$. Since $\ol{F}\in k[X]$ is separable (cf. Remark \ref{rm:defining-poly-is-separable}), then $F'(w)\not\equiv0\spmod{\pi}$. Thus Hensel's Lemma applies, that is, there exists $x_0\in\O$ such that
  \begin{equation}\label{eq:hensel-to-w}
    f(x_0)=0\quad\text{and}\quad x_0\equiv w\pmod{\pi}.
  \end{equation}
  In particular, we obtain a $K$-rational point on $C$:
  \[
    P:=(x_0,\pi^N)\in C(K).
  \]

  We claim that we can upgrade the congruence in \eqref{eq:hensel-to-w} to
  \begin{equation}\label{eq:better-congruence-for-x0-w}
    \nu(x_0-w)=2N.
  \end{equation}
  The first condition in \eqref{eq:hensel-to-w} implies that $u(x_0^2-r)(x_0^2-s)(x_0^2-t)=\pi^{2N}$, that is
  \[
    (x_0^2-r)(x_0^2-s)(x_0^2-t)\equiv0\pmod{\pi^{2N}}
  \]
  Since $W=(w,0)$ satisfies the equation of $C$, we have that $w^2=r,s$ or $t$; without loss of generality, lets assume that $w^2=r$. Hence the second condition of \eqref{eq:hensel-to-w} implies that $x_0^2\equiv r\spmod{\pi}$. In particular, $\ol{x_0}$ is a root of $\ol{F}$, which is separable, and hence $x_0^2-s,x_0^2-t\not\equiv0\spmod{\pi}$. This shows that $x_0^2\equiv w\spmod{\pi^{2N}}$. Now, if $x_0\equiv w\spmod{\pi^{2N+1}}$, then there exists $a\in \O$ such that $x_0=w+\pi^{2N+1}a$. Thus Taylor's formula would imply that
  \[
    0 = f(x_0)=f(w+\pi^{2N+1}a)= f(w)+f'(w)\pi^{2N+1}a + O(\pi^{4N+2})=\pi^{2N} + O(\pi^{2N+1}),
  \]
  which cannot happen. Thus we have \eqref{eq:better-congruence-for-x0-w}.

  Next we claim that $\s_i(P)-\s_i(W)\in \E_i(\m^N)\setminus\E_i(\m^{N+1})$, so we can then apply Lemma \ref{lemma:key-lemma}. To verify the claim, we use \eqref{eq:better-congruence-for-x0-w}. Recall that $\s_i(W)$ is 2-torsion, so $\s_i(W)$ has $y$-coordinate 0, whereas $\s_i(P)$ has a nonzero $y$-coordinate by how $\s_i$ is defined and the fact that $x_0$ is a unit. Thus $\s_i(P)\neq \pm\s_i(W)$ and the addition formula is simple to use. For simplicity, we write $P_i:=(x_i,y_i):=\s_i(P)$ and $W_i:=(w_i,0):=\s_i(W)$. The addition formula tells us that the $x$-coordinate of $P_i-W_i$ is
  \[
    x(P_i-W_i)=\left(\frac{y_i}{x_i-w_i}\right)^2-x_i-w_i.
  \]
  By the formula for $\s_1$, we have that $y_1=u\pi^N$ and
  \[
    x_1 - w_1 = u(x_0^2-v)-u(w^2-v)=u(x_0-w)(x_0+w).
  \]
  Since $x_0\equiv w\spmod{\pi}$, and $w$ is a unit, then $x_0$ is a unit and $x_0\not\equiv -w\spmod{\pi}$. Thus \eqref{eq:better-congruence-for-x0-w} implies that $\nu(x_1-w_1)=2N$ and hence $\nu(y_1/(x_1-w_1))=-N$. The final term $-x_1-w_1$ is integral and hence
  \[
    \nu(x(P_1-W_1))=-2N.
  \]
  In particular, $P_1-W_1\in \E_1(\m^N)\setminus\E_1(\m^{N+1})$ as required.

  For $i=2$, the computations are very similar. First, we see that $y_2=-urstx_0^2\pi^N$ has valuation $N$ since $u,r,s,t$ and $x_0$ are all units. Secondly, we have
  \[
    x_2-w_2 = ust(x_0^2-r)x_0^{-2}-ust(w^2-r)w^{-2}=ustr(x_0-w)(x_0+w)x_0^{-2}w^{-2},
  \]
  which again has valuation $2N$ by \eqref{eq:better-congruence-for-x0-w}. Since $-x_2-w_2$ is integral, we similarly have that
  \[
    \nu(x(P_2-W_2))=-2N.
  \]
  Hence $P_2-W_2\in \E_2(\m^N)\setminus\E_2(\m^{N+1})$.

  Now we apply Lemma \ref{lemma:key-lemma} to $P_i-W_i\in \E_i(\m^N)\setminus\E_i(\m^{N+1})$. Since $W_i$ is 2-torsion, then $2P_i\in \E_i(\m^N)\setminus\E_i(\m^{N+1})$. Next we apply Lemma \ref{prop:symbols_coming_from_scholten_curves} to the span $E_1\leftarrow C\ra E_2$ to obtain
  \[
    4\{P_1,P_2\}_{K/K}=0.
  \]
  We use bilinearity to conclude that
  \[
    \{2P_1,2P_2\}_{K/K}=0.
  \]
  Since $2P_i\in \E_i(\m^N)\setminus\E_i(\m^{N+1})$, we have that the signature of the above symbol is $(N,N)$. This concludes the proof.

\end{proof}

This method recovers previously known results about $K/K$-symbols \cite[Theorem 4.6]{GazakiHiranouchi2021} for $K=\Q_p$, via completely different methods. Theorem \ref{thm:main-theorem}, for $N=1$, and Example \ref{ex:1-1-signature} combine to give the following.

\begin{corollary}
  Let $E_1$ and $E_2$ be elliptic curves defined over $\Q_p$ with full 2-torsion and supersingular reduction. Then $\Phi_{\Q_p}=0$ in $K(\Q_p;E_1,E_2)/p$.
\end{corollary}

Theorem \ref{thm:main-theorem} produces $K/K$-symbols with diagonal signature that vanish. However, one would like to claim that all $K/K$-symbols with diagonal signature vanish. The above results are not enough to show this in full generality, but we do have the following corollary.

\begin{corollary}
  Let $K$ be a totally ramified extension of $\Q_p$, with $p>3$. Let $E_1$ and $E_2$ be elliptic curves with full 2-torsion and supersingular reduction. If the degree $e=[K:\Q_p]$ satisfies $e<p-1$, then every $K/K$-symbol in $K(K;E_1,E_2)$ of signature $(e,e)$ vanishes modulo $p$.
\end{corollary}

\begin{proof}
  Let $\{P,Q\}_{K/K}\in K(K;E_1,E_2)$ be a $K/K$-symbol of signature $(e,e)$.
  By Theorem \ref{thm:main-theorem}, there are points $P_e\in \E_1(\m^N)\setminus\E_1(\m^{N+1})$ and $Q_e\in \E_2(\m^N)\setminus\E_2(\m^{N+1})$ such that $\{P_e,Q_e\}_{K/K}\equiv0\spmod{p}$. Since $\E_1(K)/p\cong E_1(K)/p$ is an $e$-dimensional graded $\F_p$-vector space, we can find points $P_1,\ldots,P_{e-1}\in \E_1(K)$ such that $\{[P_1],\ldots,[P_e]\}$ is a basis of $\E_1(K)/p$ and such that $P_i\in\E_1(\m^i)\setminus\E_1(\m^{i+1})$.

  Therefore there are scalars $t_1,\ldots,t_e\in\F_p$ such that $[\wh{P}]=\sum t_i[P_i]$. By the properties of the signature (cf. Proposition \ref{prop:properties-of-signature}), we have that
  \[
    e = \sS([P])=\sS([\wh{P}])=\sS\left(\sum t_i[P_i]\right)=\min_{t_i\neq 0}\{\sS(t_i[P_i])\}=\min_{t_i\neq 0}\{\sS([P_i])\}=\min\{i\mid t_i\neq 0\},
  \]
  since the nonzero summands of $\sum t_i[P_i]$ all have distinct signatures. Hence we must have $t_1=\cdots=t_{e-1}=0$ and $t_e\neq 0$. We conclude that $[\wh{P}]=t_e[P_e]$ for some $t_e\neq0$. Analogously, we have that $[\wh{Q}]=s_e[Q_e]$ for some $s_e\neq0$.
  
  By the definition of the associated formal point, we have that $[P]=[\wh{P}]$ and $[Q]=[\wh{Q}]$ so that
  \[
    \{P,Q\}_{K/K}\spmod{p} =\Phi_K([P]\o[Q]) = \Phi_K([\wh{P}]\o[\wh{Q}])=\Phi_K(t_e[P_e]\o s_e [Q_e]).
  \]
  Since $\Phi_K$ is $\F_p$-linear, we have
  $\Phi_K([P]\o[Q])=t_e s_e\Phi_K([\wh{P}_e]\o[\wh{Q}_e])$. We conclude that
  \[
    \{P,Q\}_{K/K}\spmod{p}= t_e s_e \Phi_K([P_e]\o[Q_e])=t_e s_e\{P_e,Q_e\}_{K/K}\equiv0\spmod{p}.
  \]

\end{proof}

\section{The Case of Quadratic Ramified Extensions}\label{sec:computations}

In the previous section we described how one can choose a specific point on a genus 2 cover of $E_1$ and $E_2$, namely a point that is $p$-adically close to a Weierstrass point, to produce a WR relation with signature $(n,n)$. However, choosing other points yields ``off-diagonal'' signatures as well. In this section we will show computational evidence that this method usually produces enough relations in $K(K;E_1,E_2)/p$ to show that $\Symb_{K/K}=0$, in view of Proposition \ref{prop:necessary-conditions-for-PhiL-0}. 

First we describe the general procedure using SAGE and then present the results. We restrict to the case that $K$ is a totally ramified extension of $\Q_p$ with $p\equiv3\spmod{4}$ (cf. Lemma \ref{lemma:p-3mod4-and-legendre-form}). In particular, $k=\F_p$.

\vspace{3mm}
\noindent\textbf{Step 1:} Given $K$, we compute the elements $\ol{\l}\in\F_p$ for which the Legendre form elliptic curve $y^2=x(x-1)(x-\ol{\l})$ over $\F_p$ is supersingular. These correspond to roots of the supersingular polynomial
\[
  H_p(T):=\sum_{i=0}^{m}\binom{m}{i}^2 T^i,\qquad m=\frac{p-1}{2}.
\]
See Remark \ref{rm:about-ssing-red}. We then obtain a list $\{\ol{\l_1},\ldots,\ol{\l_m}\}\subset\F_p^\times$ of parameters for which $E_{\ol{\l_i}}$ is supersingular.

\vspace{3mm}
\noindent\textbf{Step 2:} For each pair $(\ol{\l},\ol{\mu})$ in the list of supersingular parameters given in the previous step, we compute the six genus 2 curves from Table \ref{table:scholten_maps-condensed} that span $E_\l$ and $E_\mu$ where $\l,\mu\in\Z$ are lifts of $\ol{\l}$ and $\ol{\mu}$. It is possible that some of these 6 curves degenerate when reduced over $k$; we will discard these curves.

\vspace{3mm}
\noindent\textbf{Step 3:} Given a span $E_\l\leftarrow C\ra E_\mu$ from Step 2, we produce $K$-rational points on $C$ in the following manner. We fix some $x_0\in K$ and solve the quadratic equation $y^2=f(x_0)$, where $y^2=f(x)$ is the defining equation of the given hyperelliptic curve. This generates the ``quadratic'' point $(x_0,\sqrt{f(x_0)})$. In general this may not have a solution, but $K$ has enough squares for this method to work reliably. Indeed, the group $K^\times/K^{\times 2}$ has order four.

We can systematically generate these ``quadratic'' points as follows. If $x_0\in K$ has valuation $\nu(x_0)=n$ for some $n\neq0$, then by the form of the defining equation $y^2=f(x)$ of the curve $C$, we have that $\nu(f(x_0))=6n$ (resp. $\nu(f(x_0))=0$) if $n<0$ (resp. $n>0$). Hence if $y_0=\sqrt{f(x_0)}$ is defined over $K$, then $\nu(y_0)=3n$ (resp. $\nu(y_0)=0$). By the formulas for the spanning maps $\s_1$ and $\s_2$, then the coordinates $(x_i,y_i):=\s_i(x_0,y_0)$ have valuations
\[
  \nu(x_1)=\begin{cases}
    2n &\text{if}\; n<0\\
    0 &\text{if}\; n>0,
  \end{cases}
  \quad\text{and}\quad
  \nu(x_2)=\begin{cases}
    0 &\text{if}\; n<0\\
    -2n &\text{if}\; n>0.
  \end{cases}
\]
This means that $\s_1(x_0,y_0)\in \E_1(\m^{|n|})\setminus\E_1(\m^{|n|})$ if $n<0$ and $\s_2(x_0,y_0)\in \E_2(\m^n)\setminus\E_2(\m^n)$ if $n>0$. We know by Lemma \ref{lemma:formal-group-mod-p} that we only need to consider the case when $|n|<e+1$. Since we can view $K$ as Laurent series in some uniformizer $\pi$, then we can take $x_0$ in the finite set of truncated Laurent series
\[
  \mathcal{B}:=\{c_{-e}\pi^{-e}+\cdots +c_0+c_1\pi+\cdots+c_e\pi^e\in K\mid c_i=0,\ldots,p-1\}.
\]
Note that this set has $p^{2e+1}$ many elements.

\vspace{3mm}
\noindent\textbf{Step 4:} Given a span $E_\l\leftarrow C\ra E_\mu$ from Step 2 and a $K$-rational point $P_0\in C(K)$ from Step 3, we compute the signature of $\{\s_1(P_0),\s_2(P_0)\}$ by computing the $\m$-adic expansions of $n_1\s_1(P)\in \E_1(\m)$ and $n_2\s_2(P)\in\E_2(\m)$ where $n_i$ is the order of the reduced point $\ol{\s_i(P)}$ on the reduced elliptic curve (see the discussion after Definition \ref{def:associated-formal-point}). From the resulting list of signatures, we check if it includes (1,1), (1,2), (2,1) and (2,2). In view of Example \ref{ex:quadratic-ramified}, this would imply that $\Symb_{K/K}=0$.

\vspace{3mm}

Putting everything together, we have the following algorithm.
\begin{algorithm}\label{alg:signatures} Let $K$ be a quadratic ramified extension of $\Q_p$ with $p>3$.
  \begin{enumerate}
    \item[Input:] A pair of integers $(\l,\mu)$ such that $E_\l$ and $E_\mu$ have supersingular reduction.
    \item[Output:] A list of signatures $(i,j)$ corresponding to WR relations of the form $\{P,Q\}_{K/K}\equiv0\spmod{p}$ in $K(K;E_1,E_2)$.
  \end{enumerate}
\end{algorithm}

\begin{example}($K=\Q_{7}(\sqrt{7})$) We have that $K$ is a ramified quadratic extension of $\Q_7$. We fix $\pi:=\sqrt{7}$ for the uniformizer. By Duering's criterion, we see that $E_\l$ is supersingular if $\l\equiv2,4,6\spmod{p}$. So let us consider the elliptic curves $E_2$ and $E_6$.
  
  The spanning genus 2 curves are given by
  \begin{gather*}
    C_1:\; y^2 = -\frac{1}{2}(x^2-5)(x^2-3)(x^2-1),\quad
    C_2:\; y^2 = \frac{2}{11}(x^2+5)(x^2-6)(x^2-\tfrac{1}{2})\\
    C_3:\; y^2 = -\frac{1}{2}(x^2+5)(x^2+3)(x^2+1),\quad 
    C_4:\; y^2 = \frac{2}{11}(x^2-5)(x^2+6)(x^2+\tfrac{1}{2})
  \end{gather*}
  The last two curves degenerate over $\F_7$. Notice that each of these curves has a $K$-rational Weierstrass point since the sextic polynomials all have at least one factor that splits (e.g. $-5$ is a square in $\Q_7$).

  The set $\mathcal{B}$ has $p^{2e+1}=7^5$ elements so we have plenty of $K$-rational points on the various $C_i$'s. Consider the following four:
  \begin{align*}
    P&=(3+\pi+3\pi^2+4\pi^4+\cdots,3+4\pi+6\pi^2+6\pi^3+\cdots)\in C_1(K),\\
    Q&=(2\pi^{-1}+6\pi+\pi^2+\pi^3+\cdots,2\pi^{-3}+6\pi^{-1}+3+\pi+3\pi^2+3\pi^3+\cdots)\in C_4(K),\\
    S&=(4\pi+2\pi^3+3\pi^4+\cdots,2+3\pi^2+\pi^4+\cdots)\in C_2(K),\\
    T&=(2\pi^2+5\pi^3+3\pi^4+\cdots,2\pi^2+3\pi^4+\cdots)\in C_1(K).
  \end{align*}
  Then one can compute the signatures to be
  \begin{align*}
    \sS(\{\s_{1}(P),\s_{2}(P)\}_{K/K})&=(1,1),\\
    \sS(\{\s_{1}(Q),\s_{2}(Q)\}_{K/K})&=(1,2),\\
    \sS(\{\s_{1}(R),\s_{2}(R)\}_{K/K})&=(2,1),\\
    \sS(\{\s_{1}(T),\s_{2}(T)\}_{K/K})&=(2,2).
  \end{align*}
  Thus by Proposition \ref{prop:necessary-conditions-for-PhiL-0} we can conclude that $\Symb_{K/K=0}$ in $K(K;E_{2},E_{6})/7$.
\end{example}

In view of the above example, given a totally ramified $p$-adic field $K$ and a pair of Legendre form elliptic curves $E_\l$ and $E_\mu$ with supersingular reduction, we say that \emph{Algorithm \ref{alg:signatures} proves $\Symb_{K/K}=0$}, or that it is \emph{successful}, if the output of the algorithm includes the signatures (1,1), (1,2), (2,1) and (2,2). We run Algorithm \ref{alg:signatures} for all totally ramified quadratic extensions of $\Q_p$ for the first ten primes $p\equiv3\spmod{4}$ and check the success. The results are in table \ref{table:signature-statistics-for-quadratic-exts} below.

The algorithm fails to be 100\% successful for $p=23,71$. However, this does not prove that the proposed method to produce WR relations will not work. For example, this algorithm only uses six genus 2 curves to produce WR, where there are infinitely many hyperelliptic covers one can use by \cite{GazakiLove2023}.

\begin{table}
  \caption{Results of Algorithm \ref{alg:signatures} applied to all quadratic ramified extensions of $\Q_p$ for the first 10 primes $p\equiv3\spmod{4}$}
  \label{table:signature-statistics-for-quadratic-exts}
  \begin{center}
    {\renewcommand{\arraystretch}{1.3}
    \begin{tabular}{|c|c|c||c|c|c|}
      \hline
      $p$ & $K$ & $\begin{gathered}
      \text{Success rate of}\\ \text{Algorithm \ref{alg:signatures}}
      \end{gathered}$ & $p$ & $K$ & $\begin{gathered}
      \text{Success rate of}\\ \text{Algorithm \ref{alg:signatures}}
      \end{gathered}$\\
      \hline
      \multirow{2}{*}{$7$} & $\Q_7(\sqrt{7})$ & $9/9=100\%$ &\multirow{2}{*}{$43$} & $\Q_{43}(\sqrt{43})$ & $9/9=100\%$ \\ \cline{2-3}\cline{5-6}
      & $\Q_7(\sqrt{-7})$ & $9/9=100\%$ &  & $\Q_{43}(\sqrt{-43})$ & $9/9=100\%$ \\ \hline

      \multirow{2}{*}{$11$} & $\Q_{11}(\sqrt{11})$ & $9/9=100\%$ &\multirow{2}{*}{$47$} & $\Q_{47}(\sqrt{47})$ & $225/225=100\%$ \\ \cline{2-3}\cline{5-6}
      & $\Q_{11}(\sqrt{-11})$ & $9/9=100\%$ &  & $\Q_{47}(\sqrt{-47})$ & $225/225=100\%$ \\ \hline

      \multirow{2}{*}{$19$} & $\Q_{19}(\sqrt{19})$ & $9/9=100\%$ &\multirow{2}{*}{$59$} & $\Q_{59}(\sqrt{59})$ & $81/81=100\%$ \\ \cline{2-3}\cline{5-6}
      & $\Q_{19}(\sqrt{-19})$ & $9/9=100\%$ &  & $\Q_{59}(\sqrt{-59})$ & $81/81=100\%$ \\ \hline

      \multirow{2}{*}{$23$} & $\Q_{23}(\sqrt{23})$ & $63/81=77.77\%$ &\multirow{2}{*}{$67$} & $\Q_{67}(\sqrt{67})$ & $9/9=100\%$ \\ \cline{2-3}\cline{5-6}
      & $\Q_{23}(\sqrt{-23})$ & $63/81=77.77\%$ &  & $\Q_{67}(\sqrt{-67})$ & $9/9=100\%$ \\ \hline

      \multirow{2}{*}{$31$} & $\Q_{31}(\sqrt{31})$ & $81/81=100\%$ &\multirow{2}{*}{$71$} & $\Q_{71}(\sqrt{71})$ & $404/441=91.61\%$ \\ \cline{2-3}\cline{5-6}
      & $\Q_{31}(\sqrt{-31})$ & $81/81=100\%$ &  & $\Q_{71}(\sqrt{-71})$ & $405/441=91.83\%$ \\ \hline
    \end{tabular}
    }
  \end{center}
\end{table}


\bibliography{bibliography}
\bibliographystyle{alpha}


\end{document}